\documentclass[a4paper,12pt,reqno]{amsart}
\usepackage{amssymb, amsmath, amsfonts, amscd}
\usepackage{mathrsfs, latexsym}
\usepackage[all,ps,cmtip]{xy}
\usepackage{array, longtable}
\usepackage{bm, enumitem}
\usepackage{xcolor}

\title{On the bicanonical map of algebraic threefolds of general type}
\date{\today}

\dedicatory{Dedicated to the memory of professor Zhihua Chen}

\author{Meng Chen}
\address{\rm School of Mathematical Sciences, Fudan University, Shanghai 200433, China}
\email{mchen@fudan.edu.cn}

\author{Chen Jiang}
\address{\rm Shanghai Center for Mathematical Sciences \& School of Mathematical Sciences, Fudan University Shanghai 200438, China}
\email{chenjiang@fudan.edu.cn}

\author{Jianshi Yan}
\address{\rm Department of Mathematics, Northeastern University, Shenyang 110819, China}
\email{yanjs@mail.neu.edu.cn}

\thanks{This project is supported by was supported by NSFC for Innovative Research Groups \#12121001, National Key Research and Development Program of China \#2023YFA1010600, and National Key Research and Development Program of China \#2020YFA0713200. The first author is supported by NSFC Program (\#12071078). The third author is supported by the Fundamental Research Funds for the Central Universities (\#2205012). The first two authors are members of LMNS, Fudan University}

\newcommand{\bQ}{{\mathbb Q}}
\newcommand{\bP}{{\mathbb P}}
\newcommand{\roundup}[1]{\lceil{#1}\rceil}
\newcommand{\rounddown}[1]{\lfloor{#1}\rfloor}

\newcommand\OO{{\mathcal{O}}}

\newcommand{\Mov}{\operatorname{Mov}}

\newcommand{\lct}{\text{\rm lct}}
\newcommand{\glct}{\text{\rm glct}}

\newtheorem{thm}{Theorem}[section]
\newtheorem{lem}[thm]{Lemma}

\newtheorem{prop}[thm]{Proposition}

\theoremstyle{definition}
\newtheorem{defn}[thm]{Definition}

\newtheorem{question}[thm]{Question}

\newtheorem{notation}[thm]{Notation}

\theoremstyle{remark}

\begin{document}
\begin{abstract} We answer an open problem of the first author and Zhang (see Open Problem 6.4 (3) in Math. Z. 258 (2008), 565--585) and prove that, for any nonsingular projective $3$-fold of general type with the geometric genus greater than 201,
the bicanonical system is not composed of any pencil of surfaces.
\end{abstract}
\maketitle

\pagestyle{myheadings}
\markboth{\hfill M. Chen, C. Jiang, J. Yan\hfill}{\hfill \hfill}
\numberwithin{equation}{section}


\section{Introduction}

In birational geometry, studying the behavior of pluricanonical maps is usually a very important step towards the classification of varieties. Given a nonsingular projective variety $V$ of general type, the geometry induced from the $m$-canonical system $|mK_V|$ is connected with various aspects, such as the birationality of $\Phi_{|mK_V|}$, the value of $\dim \overline{\Phi_{|mK_V|}(V)}$, the value of $P_{m}(V)$ for any positive integer $m$ and so on.

This note concentrates on whether the bicanonical system $|2K_V|$ is composed of a pencil or not, where $V$ is a $3$-fold of general type. In fact, the first author and Zhang raised the following open problem in 2008:

\begin{question}\label{op} (\cite[Open Problem 6.4 (3)]{CZ08}) Is there a constant $N>0$ such that the bicanonical system $|2K_V|$ is not composed of any pencil of surfaces for every nonsingular projective $3$-fold $V$ of general type with $p_g(V)\geq N$?
\end{question}

As far as we know, there were two relevant results by the first author:
\begin{itemize}
\item It is proved in \cite{CAM98} that, for any smooth minimal $3$-fold $X$ of general type, $|mK_X|$ is not composed of a pencil for all $m\geq 3$.

\item It is proved in \cite{OSAKA} that, for any Gorenstein minimal $3$-fold $X$ of general type, if $|2K_X|$ is composed of a pencil of surfaces $\{F\}$, then one must have $p_g(F)=1$ and $K_{F_0}^2\leq 3$ where $F_0$ is the minimal model of $F$.
\end{itemize}

The main purpose of this paper is to give a positive answer to Question \ref{op} and prove the following theorem.

\begin{thm}\label{thm1}
 Let $V$ be a smooth projective $3$-fold of general type. If $p_g(V)\geq 202$, then $|2K_V|$ is not composed of any pencil of surfaces, that is, $\dim \overline{\Phi_{|2K_V|}(V)}\geq 2$.\end{thm}

\section{Preliminaries}

Throughout this paper, we work over any algebraically closed field of characteristic $0$. We adopt the standard notation and definitions in \cite{KMM} and \cite{K-M}, and will freely use them.

A projective variety $X$ is said to be {\it of general type} if, for some resolution $W\rightarrow X$, $K_W$ is big. A projective variety $X$ is said to be {\it minimal} if $X$ has at worst $\bQ$-factorial terminal singularities and $K_X$ is nef.

\subsection{Rational maps defined by linear systems of Weil divisors}\label{b setting}\

Let $X$ be a normal projective $3$-fold.
Consider a $\bQ$-Cartier Weil divisor $D$ on $X$ with $h^0(X, D)\geq 2$. We study the rational map $\Phi_{|D|}$ defined by $|D|$, say
$$X\overset{\Phi_{|D|}}{\dashrightarrow} \bP^{h^0(D)-1}$$ which is
not necessarily well-defined everywhere. By Hironaka's desingularization theorem, we can take a projective birational morphism $\pi: W\rightarrow X$ such that:
\begin{itemize}
\item [(i)] $W$ is smooth projective;
\item [(ii)] the movable part $|M|$ of the linear system
$|\rounddown{\pi^*(D)}|$ is basepoint-free and, consequently,
the rational map $\gamma=\Phi_{|D|}\circ \pi$ is a projective morphism;
\item [(iii)] the support of the
union of $\pi_*^{-1}(D)$ and the exceptional divisors of $\pi$ is of
simple normal crossings.
\end{itemize}
Let $W\overset{f}\longrightarrow \Gamma\overset{s}\longrightarrow \Sigma$
be the Stein factorization of $\gamma$ with $\Sigma=\gamma(W)\subseteq
\bP^{h^0(D)-1}$. We have the following commutative
diagram:
$$\xymatrix@=4.5em{
W\ar[d]_\pi \ar[dr]^{\gamma} \ar[r]^f& \Gamma\ar[d]^s\\
X \ar@{-->}[r]^{\Phi_{|D|}} & \Sigma}
$$



If $\dim(\Gamma)=1$, then $\Gamma$ is a smooth curve and
a general fiber $F$ of $f$ is a smooth projective surface
by Bertini's theorem. In this case, we say that
$|D|$ {\it is composed of a pencil of surfaces}. 

\subsection{Global log canonical thresholds}\

We recall the concept of global log canonical thresholds for minimal  varieties of general type from \cite{Noether}.

Let $X$ be a normal variety with lc singularities and $D\geq 0$ a $\bQ$-Cartier $\bQ$-divisor. The {\it log canonical threshold} of $D$ with respect to $X$ is defined by
$$\lct(X; D) = \sup\{t\geq 0 \mid (X, tD) \text{ is lc}\}.$$

\begin{defn}
Let $Y$ be a normal projective variety with at worst klt singularities such that $K_Y$ is nef and big. The {\it global log canonical threshold} ({\it \glct}, for short) of $Y$ is defined as the following:
\begin{align*}
\glct(Y){}&=\inf\{\lct(Y; D)\mid 0\leq D\sim_\bQ K_Y\}\\
{}&=\sup\{t\geq 0\mid (Y, tD) \text{ is lc for all }0\leq D\sim_\bQ K_Y\}.
\end{align*}
\end{defn}

\section{GLCT of minimal surfaces of general type with $p_g\leq 1$ and $c_1^2\leq 8$}

In this section, we give an estimate for glct of minimal surfaces of general type with $p_g\leq 1$ and $c_1^2\leq 8$.
We will use Koll\'ar's method and the notation in \cite[Appendix]{Noether}.

First we recall a local invariant defined by Koll\'ar, called
 the {\it minimal multiplier codimension}
$$
\textrm{mcd}(c):=\min_{G, \Delta} \Big\{\dim\Big(\big(\mathbb{C}[x,y]/{\mathcal J}^+(\Delta)\big)^G\Big)\Big\},
$$
where  $G$ runs though all finite subgroups of $\textrm{SL}_2(\mathbb{C})$,   $\Delta$ runs through all $G$-invariant divisors such that $\lct_0(\mathbb{C}^2; \Delta)<c$ (where $\lct_0$ means the log canonical threshold in a neighborhood of the origin) and
we use the upper multiplier ideal
${\mathcal J}^+(\Delta):= {\mathcal J}((1-\epsilon)\Delta)$
for $0<\epsilon\ll 1$ \cite[9.2.1]{Positivity2}. 

\begin{prop}[{cf. \cite[Proposition~A.3]{Noether}}]\label{prop:A3}
 Let $S$ be a projective surface with Du~Val singularities and $H$ an ample Cartier divisor. Let $\Delta$ be an effective $\mathbb{Q}$-divisor such that $\Delta\equiv H$. Then
$$h^0(S, \OO_S(K_S+H))\geq \operatorname{mcd}(\lct(S;\frac{1}{H^2}\Delta)).$$
\end{prop}
\begin{proof}
This was indicated by Koll\'ar \cite[Remark~A.7.2]{Noether}. We go through the proof of \cite[Proposition~A.3]{Noether}. In fact, the coefficients of $\Delta$ are at most $(\Delta\cdot H)=H^2$. So $\lfloor{\frac{1-\epsilon}{H^2}\Delta}\rfloor=0$ and
$(S, \frac{1-\epsilon}{H^2}\Delta)$ has only isolated non-log-canonical centers for $0<\epsilon<1$. The rest of the proof is exactly the same as \cite[Proposition~A.3]{Noether}.
\end{proof}

\begin{prop}\label{prop:glct>1/200}
 Let $S$ be a minimal surface of general type with $p_g(S)\leq 1$ and $K^2_S\leq 8$. Then $\glct(S)\geq \frac{1}{200}.$
\end{prop}

\begin{proof}
Fix an effective $\bQ$-divisor $B\sim_\bQ K_S$, it suffices to show that $\lct(S; B)\geq \frac{1}{200}$.
Denote by $S'$ the canonical model of $S$ and $\tau: S \to S'$ the induced map. Consider the effective $\bQ$-divisor $\Delta:= \tau_*B\sim_\bQ K_{S'}$. Since $\tau$ is crepant, $\lct(S; B)=\lct(S'; \Delta)$.
Indeed, for any $t\geq 0$, we have 
$$
K_S+tB=\tau^*(K_{S'}+t\Delta)
$$
and hence $(K_S, tB)$ is lc if and only if 
$(K_{S'}, t\Delta)$ is lc. So we only need to bound $\lct(S';\Delta)$.

Applying Proposition~\ref{prop:A3} to $H=K_{S'}$, we get
\begin{align}
 h^0( S', \OO_{S'}(2K_{S'}))\geq \operatorname{mcd}(\lct(S';\frac{1}{K_{S'}^2}\Delta)). \label{eq:h0>mcd}
\end{align}
On the other hand, by the Riemann--Roch formula and the Kodaira vanishing theorem,
$$
h^0( {S'}, \OO_{S'}(2K_{S'}))=K_{S'}^2+\chi(S',\OO_{S'})\leq K_{S'}^2+1+p_g({S'})\leq 10.
$$
So by \eqref{eq:h0>mcd} and \cite[Lemma~A.4]{Noether},
we get
$$\lct({S'};\frac{1}{K_{S'}^2}\Delta)\geq \frac{1}{25}.$$
So
$$\lct(S';\Delta)=\frac{1}{K_{S'}^2}\lct(S;\frac{1}{K_{S'}^2}\Delta)\geq \frac{1}{200}.$$
\end{proof}

\section{The geometry of bicanonical maps}

\begin{notation}\label{notation1}
 Let $X$ be a minimal $3$-fold of general type with $p_g(X)\geq 3$. Assume that $|K_X|$ is composed of a pencil of surfaces.
Keep the notation in Subsection~\ref{b setting} with $D=K_X$.
We have the following commutative
diagram:
$$\xymatrix@=4.5em{
W\ar[d]_\pi \ar[dr]^{\gamma} \ar[r]^f& \Gamma\ar[d]^s\\
X \ar@{-->}[r]^{\Phi_{|K_X|}} & \Sigma}
$$
Then we may write
$$
\pi^*K_X\sim_{\bQ} M+Z \equiv aF+Z
$$
where
\begin{itemize}
 \item $M=\Mov|\rounddown{\pi^*K_X}|$;
 \item $Z$ is an effective $\mathbb{Q}$-divisor;
 \item $F$ is a general fiber of $f$; and
 \item $a\geq p_g(X)-1\geq 2$.
\end{itemize}
Denote by $\sigma: F\to F_0$ the contraction morphism onto the minimal model $F_0$. Taking the restriction on $F$, we have
\begin{align}
 \pi^*K_X|_F\sim Z|_F.\label{eq:K|F=Z|F}
\end{align}
Here although both sides are $\mathbb{Q}$-divisors, the linear equivalence makes sense as the difference $\pi^*K_X|_F- Z|_F$ is a principal divisor.
\end{notation}

\begin{lem}\label{lem1} Keep the setting in Notation~\ref{notation1}.
 If
 \begin{align}
 h^0(F, K_F+\roundup{\pi^*K_X|_F-\frac{1}{a}Z|_F})\geq 2,\label{eq: K+ru>=2}
 \end{align}
 then $|2K_X|$ is not composed of a pencil of surfaces.
\end{lem}
\begin{proof}
 Since
$$
\pi^*K_X-\frac{1}{a}Z-F\equiv (1-\frac{1}{a})\pi^*K_X
$$
is nef and big,
by the Kawamata--Viehweg vanishing theorem (\cite{KV,VV}), we have
$$
H^1(W, K_W+\roundup{\pi^*K_X-\frac{1}{a}Z-F})=0.
$$
So there exists a natural surjective map
\begin{align}
{}&H^0(W, K_W+\roundup{\pi^*K_X-\frac{1}{a}Z})\notag\\
\to {}& H^0(F, K_F+\roundup{\pi^*K_X-\frac{1}{a}Z-F}|_F)\label{eq:surj1}\\
={}& H^0(F, K_F+\roundup{\pi^*K_X|_F-\frac{1}{a}Z|_F}).\notag
\end{align}
Here in the last equality, the restriction and the roundup commute because 
$\pi^*K_X-\frac{1}{a}Z-F$ has simple normal crossing support. 

On the other hand,
\begin{align*}
 {}& H^0(W, K_W+\roundup{\pi^*K_X-\frac{1}{a}Z})\\
 \subseteq {}& H^0(W, K_W+\roundup{\pi^*K_X})\\
 \cong {}& H^0(X, 2K_X).
\end{align*}
So the surjectivity of \eqref{eq:surj1} implies that, for a general fiber $F$, $|2K_X||_F$ has dimension at least $1$, which implies that $|2K_X|$ is not composed of a pencil of surfaces.
\end{proof}

\begin{prop}\label{prop:K^2>=9}
Keep the setting in Notation~\ref{notation1}.
If $p_g(X)\geq 35$ and
$K_{F_0}^2\geq 9$, then 
$$h^0(F, K_F+\roundup{\pi^*K_X|_F-\frac{1}{a}Z|_F})\geq 2.$$
\end{prop}

\begin{proof}
Denote $$L:=\pi^*K_X|_F-\frac{1}{a}Z|_F\equiv (1-\frac{1}{a})\pi^*K_X|_F,$$
which is a nef and big $\mathbb{Q}$-divisor on $F$.
By \cite[Corollary~2.3]{Noether},
$\pi^*K_X|_F-\frac{a}{a+1}\sigma^*K_{F_0}$ is $\mathbb{Q}$-effective.
Since $a\geq 34$ and $K_{F_0}^2\geq 9$, we have
$$
L^2=(1-\frac{1}{a})^2(\pi^*K_X|_F)^2\geq (1-\frac{1}{a})^2(\frac{a}{a+1})^2K_{F_0}^2>8.$$

If for any irreducible curve $C$ passing through a very general point $P$, we have $(L\cdot C)\geq 4$, then $K_F+\roundup{L}$ defines a biratinoal map by \cite[Proposition~4]{Mas99} or \cite[Lemma~2.5]{Chen14}.
In particular,
$$h^0(F, K_F+\roundup{\pi^*K_X|_F-\frac{1}{a}Z|_F})\geq 2.$$

Now suppose that there exists an irreducible curve $C$ passing through a very general point $P$ such that $(L\cdot C)< 4$. Since $P$ is very general, $C$ is a nef curve and its geometric genus is at least $2$. Moreover, after taking a higher model of $W$, we may assume that $C$ is smooth (note that this does not change $L^2$ and $(L\cdot C)$ by the projection formula). Then, by \cite[Theorem~2.2.15]{Positivity1}, there exists a sufficiently small positive number $\epsilon$ such that $L-(1+\epsilon)C$ is big. In particular, we may write 
$$L\sim_{\mathbb{Q}}(1+\epsilon)C+T$$ where $T$ is an effective $\mathbb{Q}$-divisor.
Then, by the Kawamata--Viehweg vanishing theorem (\cite{KV,VV}),
$$
H^1(F, K_F+\roundup{L-C-\frac{1}{1+\epsilon}T})=0,
$$
which implies that the natural map
$$
H^0(F, K_F+\roundup{L-\frac{1}{1+\epsilon}T})\to H^0(C, K_C+\roundup{L-\frac{1}{1+\epsilon}T-C}|_C)
$$
is surjective.
Since $\deg(\roundup{L-\frac{1}{1+\epsilon}T-C}|_C)>0$, it is clear that
$$h^0(C, K_C+\roundup{L-\frac{1}{1+\epsilon}T-C}|_C)\geq 2$$ by the Riemann--Roch formula. Therefore, $
H^0(F, K_F+\roundup{L})\geq 2$.
\end{proof}

Recall the following result from \cite{Noether, Noether_Add} connecting the geometry of $|K_X|$ with $\glct(F_0)$.
\begin{prop}[{cf. \cite[Corollary~3.3]{Noether}, \cite[Corollary~3]{Noether_Add}}]\label{prop:Mov free}
Keep the setting in Notation~\ref{notation1}.
If 
$p_g(X)\geq \frac{1}{\glct(F_0)}+1$, then there exists a minimal  $3$-fold $Y$, being birational to $X$, such that $\text{\rm Mov}|K_Y|$ is basepoint-free.
\end{prop}
\begin{proof}
This is a direct consequence of \cite[Proposition~2]{Noether_Add} and \cite[Lemma~3.2]{Noether}.
\end{proof}

\begin{prop}\label{prop:pg>1/glct+1}
Keep the setting in Notation~\ref{notation1}.
 If $$p_g(X)> \frac{1}{\glct(F_0)}+1,$$ then $$h^0(F, K_F+\roundup{\pi^*K_X|_F-\frac{1}{a}Z|_F})\geq 2.$$
\end{prop}

\begin{proof}
By Proposition~\ref{prop:Mov free}, after possibly replacing $X$ by another minimal model, we may assume that $\text{\rm Mov}|K_X|$ is basepoint-free. Then by construction, $X\to \Sigma$ is a morphism 
and we have the following commutative diagram:
 $$\xymatrix@=4.5em{ W\ar[d]_\pi \ar[r]^f& \Gamma\ar[d]^s\\ X \ar[ur]^{f_0}\ar[r]^{\Phi_{|K_X|}} & \Sigma} $$
In this case, a general fiber of $f_0$ is a minimal surface of general type. In other words, for a general fiber $F$ of $f$, its image on $X$ is just the minimal model $F_0$ of $F$ and $\pi|_F$ is just $\sigma$. In particular, $\pi^*K_X|_F\sim \sigma^*K_{F_0}$. Combining with \eqref{eq:K|F=Z|F}, we have $Z|_F\sim \sigma^*K_{F_0}$. Denote $B:=\sigma_*(Z|_F)$, then $B$ is an effective divisor such that $B\sim K_{F_0}$ and $\sigma^*B=Z|_F$. 

As $a>\frac{1}{\glct(F_0)}$, $(F_0, \frac{1}{a}B)$ is klt. This implies that
$$
\sigma_*\mathcal{O}_F(K_F+\roundup{-\sigma^*K_{F_0}-\frac{1}{a}Z|_F})=\mathcal{O}_{F_0}.
$$
Tensoring with $\mathcal{O}_{F_0}(2K_{F_0})$ and by the projection formula, we have
$$
\sigma_*\mathcal{O}_F(K_F+\roundup{\sigma^*K_{F_0}-\frac{1}{a}Z|_F})=\mathcal{O}_{F_0}(2K_{F_0}).
$$
Here we learnt this trick from \cite[Proof~of~Lemma~3.14]{HZ}.

So
\begin{align*}
 H^0(F, K_F+\roundup{\pi^*K_X|_F-\frac{1}{a}Z|_F})\cong H^0(F_0, 2K_{F_0}).
\end{align*}
Finally, by the Riemann--Roch formula, we have
$$h^0(F_0, 2K_{F_0})=K_{F_0}^2+\chi(\mathcal{O}_{F_0})\geq 2.$$
\end{proof}

\begin{proof}[Proof of Theorem~\ref{thm1}] We may replace $V$ by its minimal model $X$. Assume, to the contrary, that $|2K_X|$ is composed of a pencil of surfaces. We will deduce a contradiction.

In this case, $|K_X|$ is also composed of a pencil of surfaces. So $X$ satisfies the assumptions in Notation~\ref{notation1}.

If $p_g(F)\geq 2$, then \eqref{eq: K+ru>=2} holds automatically as $$\pi^*K_X|_F-\frac{1}{a}Z|_F\sim (1-\frac{1}{a})Z|_F$$ is effective.

If $K_{F_0}^2\geq 9$, then \eqref{eq: K+ru>=2} holds by Proposition~\ref{prop:K^2>=9}.

If $p_g(F)\leq 1$ and $K_{F_0}^2\leq 8$, then $\glct(F_0)\geq \frac{1}{200}$ by Proposition~\ref{prop:glct>1/200}, and hence \eqref{eq: K+ru>=2} holds by Proposition~\ref{prop:pg>1/glct+1}.

In summary, \eqref{eq: K+ru>=2} always holds and hence we conclude the theorem by Lemma~\ref{lem1}.\end{proof}
\section*{\bf Acknowledgments}
The authors appreciate effective discussions with Zhi Jiang during the preparation of this paper.


\begin{thebibliography}{99}














\bibitem{Noether} J.~A.~Chen, M.~Chen, C.~Jiang, {\em The Noether inequality for algebraic $3$-folds}, Duke Math. J. {\bf 169}(2020), no.9, 1603--1645.


\bibitem{Noether_Add} J.~A.~Chen, M.~Chen, C.~Jiang, {\em Addendum to ``The Noether inequality for algebraic $3$-folds''}, Duke Math. J. {\bf 169} (2020), no. 11, 2199--2204.



\bibitem{CAM98} M. Chen, {\em A theorem on pluricanonical maps of nonsingular minimal threefolds of general type}. Chinese Ann. Math. Ser. B {\bf 19} (1998), no. 4, 415--420.

\bibitem{OSAKA} M. Chen, {\em On pluricanonical maps for threefolds of general type. II}. Osaka J. Math. {\bf 38} (2001), no. 2, 451--468.

\bibitem{Chen14} M. Chen, {\em Some birationality criteria on $3$-folds with $p_g>1$}, Sci. China Math. {\bf 57} (2014), no. 11, 2215--2234.


\bibitem{CV05} M.~Chen, E.~Viehweg, {\it Bicanonical and adjoint linear systems on surfaces of general type}, Pacific J. Math. {\bf 219} (2005), no. 1, 83--95.

\bibitem{CZ08} M.~Chen, D.-Q.~Zhang, {\em Characterization of the 4-canonical birationality of algebraic threefolds}. Math. Z. {\bf 258} (2008), no. 3, 565--585.


















\bibitem{HZ} Y.~Hu, T.~Zhang, {\it Algebraic threefolds of general type with small volume}, arXiv:2204.02222.
















\bibitem{KV} Y. Kawamata, {\em A generalization of Kodaira-Ramanujam's vanishing theorem}. Math. Ann. {\bf 261} (1982), no. 1, 43--46.

\bibitem{KMM} Y. Kawamata, K. Matsuda, K. Matsuki, {\em Introduction
to the minimal model problem}, Algebraic geometry, Sendai, 1985, 283--360,
Adv. Stud. Pure Math., 10, North-Holland, Amsterdam, 1987.



\bibitem{K-M} J.~Koll{\'a}r, S.~Mori, {\em Birational geometry of algebraic varieties}, Cambridge Tracts in Mathematics {\bf 134}, Cambridge Univ. Press, 1998.


\bibitem{Positivity1} R.~Lazarsfeld, \emph{Positivity in algebraic geometry, I, Classical setting: line bundles and linear series}, Results in Mathematics and Related Areas. 3rd Series. A Series of Modern Surveys in Mathematics, 48. Springer-Verlag, Berlin, 2004.

\bibitem{Positivity2} R.~Lazarsfeld, \emph{Positivity in algebraic geometry, II, Positivity for vector bundles, and multiplier ideals}, Results in Mathematics and Related Areas. 3rd Series. A Series of Modern Surveys in Mathematics, 49. Springer-Verlag, Berlin, 2004. 


\bibitem{Mas99} V. Ma\c{s}ek, {\it Very ampleness of adjoint linear systems on smooth surfaces with boundary}, Nagoya Math. J. {\bf 153} (1999), 1--29.


\bibitem{VV} E. Viehweg, {\em Vanishing theorems}. J. Reine Angew. Math. {\bf 335} (1982), 1--8.





\end{thebibliography}
\end{document}